\documentclass[11pt]{amsart}
\usepackage{pgf,tikz,pgfplots,color}
\pgfplotsset{compat=1.15}
\usepackage{mathrsfs} 
\usetikzlibrary{arrows}
\usepackage{fancyhdr}
\usepackage{lastpage} 

\usepackage[margin=1in]{geometry}
\usepackage{geometry}
\usepackage{hyperref}
\hypersetup{
    colorlinks=true,
    linkcolor=blue,
    citecolor=brown,
    filecolor=magenta,
    urlcolor=blue,
}
\usepackage{amsmath}
\usepackage{amsfonts}
\usepackage{amssymb}
\usepackage{amsthm}
\usepackage{setspace}
\usepackage{inputenc}
\usepackage{comment}
\usepackage{mathtools}
\usepackage[shortlabels]{enumitem}

\newtheorem{thm}{Theorem}

\newtheorem{cor}[thm]{Corollary}
\newtheorem{lem}[thm]{Lemma}
\newtheorem{prop}[thm]{Proposition}

\newtheorem{ques}[thm]{Question}

\newenvironment{cor*}[2][Corollary]{\begin{trivlist}
		\item[\hskip \labelsep {\bfseries #1}\hskip \labelsep {\bfseries #2.}]}{\end{trivlist}}
\newenvironment{prop*}[2][Proposition]{\begin{trivlist}
		\item[\hskip \labelsep {\bfseries #1}\hskip \labelsep {\bfseries #2.}]}{\end{trivlist}}

\theoremstyle{definition}
\newtheorem{defn}[thm]{Definition}

\newtheorem*{defn*}{Definition}

\theoremstyle{remark}
\newtheorem{rmk}[thm]{Remark}
\newtheorem*{ack}{Acknowledgment}

\allowdisplaybreaks

\newcommand{\1}{\mathbf{1}}
\newcommand{\As}{\mathscr{A}}
\newcommand{\Bs}{\mathscr{B}}
\newcommand{\C}{\mathbb{C}}
\newcommand{\Cf}{\mathfrak{C}}
\newcommand{\Fs}{\mathscr{F}}
\newcommand{\Ff}{\mathfrak{F}}

\newcommand{\N}{\mathbb{N}}

\newcommand{\R}{\mathbb{R}}

\newcommand{\Ss}{\mathscr{S}}

\newcommand{\Z}{\mathbb{Z}}

\newcommand{\eps}{\varepsilon}

\newcommand{\loc}{\mathrm{loc}}

\newcommand{\supp}{\operatorname{supp}}

\title{A Linear Operator bounded in all Besov but not in Triebel-Lizorkin Spaces}
\author[]{Liding Yao} 
\address{Department of Mathematics,
	The Ohio State University, Columbus, OH 43210} 
\email{yao.1015@osu.edu}


\begin{document}

\begin{abstract}
    We construct a linear operator $T:\mathscr S'(\mathbb R^n)\to \mathscr S'(\mathbb R^n)$ such that $T:\mathscr B_{pq}^s(\mathbb R^n)\to\mathscr B_{pq}^s(\mathbb R^n)$ for all $0<p,q\le\infty$ and $s\in\mathbb R$, but $T(\mathscr F_{pq}^s(\mathbb R^n))\not\subset \mathscr F_{pq}^s(\mathbb R^n)$ unless $p=q$. As a result Triebel-Lizorkin spaces cannot be interpolated from Besov spaces unless $p=q$. In the appendix we purpose a question for the interpolation framework via structured Banach spaces.
\end{abstract}

\maketitle

\section{Introduction and Main Result}
It is well known that Besov spaces are real interpolation spaces to Triebel-Lizorkin spaces, since we have $(\Fs_{pq_0}^{s_0}(\R^n),\Fs_{pq_1}^{s_1}(\R^n))_{\theta,q}=\Bs_{pq}^{(1-\theta)s_0+\theta s_1}(\R^n)$ and $(\Fs_{\infty\infty}^{s_0}(\R^n),\Fs_{\infty\infty}^{s_1}(\R^n))_{\theta,q}=\Bs_{\infty q}^{(1-\theta)s_0+\theta s_1}(\R^n)$ for all $0<\theta<1$, $s_0\neq s_1$, $p\in(0,\infty)$ and $q_0,q_1,q\in(0,\infty]$. See e.g. \cite[Theorem~2.4.2]{TriebelTheoryOfFunctionSpacesI}. As a result if we have a linear operator that is bounded in all Triebel-Lizorkin spaces, then it is automatically bounded in all Besov spaces as well. 

In this paper we show that the converse is false.
\begin{thm}\label{Thm::TBdd}
    Let $(\phi_j)_{j=0}^\infty$ be a Littlewood-Paley family that defines the norms for Besov and Triebel-Lizorkin spaces (see \eqref{Eqn::BsNorm} \eqref{Eqn::TLNorm1} \eqref{Eqn::TLNorm2} below). Let $(y_j)_{j=1}^\infty\subset\R^n$ be a sequence such that $\inf_{j\neq k}|y_j-y_k|>0$. Set $\tau_{y_j}f(x):=f(x-y_j)$ and we define
    \begin{equation}\label{Eqn::DefT}
        Tf:=\sum_{j=1}^\infty\tau_{y_j}(\phi_j\ast f)=\sum_{j=1}^\infty(\phi_j\ast f)(\cdot-y_j).
    \end{equation}
\begin{enumerate}[(i)]
    \item\label{Item::TBdd::Ss} As a side result $T:\Ss'(\R^n)\to\Ss'(\R^n)$ is bounded linear if and only if there is a $N_0>0$ such that $|y_j|\le 2^{N_0j}$ for every $j\ge1$.
    \item\label{Item::TBdd::Bs} $T$ defines a bounded linear operator on Besov spaces $T:\Bs_{pq}^s(\R^n)\to\Bs_{pq}^s(\R^n)$ for all $s\in\R$ and $0<p,q\le\infty$.
    \item\label{Item::TBdd::Fs} However on Triebel-Lizorkin spaces $T(\Fs_{pq}^s(\R^n))\not\subset\Fs_{pq}^s(\R^n)$ whenever $p\neq q$.

\end{enumerate}
\end{thm}

As an immediate corollary we see that Triebel-Lizorkin spaces cannot be interpolated from Besov space (unless $p=q$) in the classical setting.
\begin{cor}\label{Cor::Interpo}
    Elements in $\{\Fs_{pq}^s(\R^n):s\in\R,\ 0<p,q\le\infty,\ p\neq q\}$ can never be any classical interpolation space from any pair of elements in $\{\Bs_{pq}^s(\R^n):s\in\R,\ 0<p,q\le\infty\}$.
\end{cor}

This seems to be a well-known result, as there are discussions on real-interpolation of Besov spaces, e.g. \cite{KrepkogorskiiInterpolation,DeVorePopovInterpolation}. But to the best of author's knowledge, Corollary~\ref{Cor::Interpo} is not found in literature.

For completeness we give more concrete statements in Corollaries~\ref{Cor::InterpoReform1} and \ref{Cor::InterpoReform2} by recalling the definitions of (both categorical and set-theoretical) interpolation spaces in Appendix \ref{Section::DefInt}.

\begin{rmk}
    Here Corollary~\ref{Cor::Interpo} is stated ``without extra structures imposed on Besov spaces''. It is important to point out that the Triebel-Lizorkin spaces can may still be obtained via interpolations on spaces with extra structures (see \eqref{Eqn::StrBana::FfromB}). This idea was brought up by Kunstmann \cite{KunstmannInterpolation} using \textit{structured Banach spaces} and later generalized by Lindermulder-Lorist \cite{LindermulderLoristInterpolation} who unified Kunstmann's approach and the classical real and complex interpolations. We will discuss it briefly in Appendix \ref{Section::StrBana}. However, a general interpolation framework is still not known to extract Triebel-Lizorkin spaces from ``structured Besov spaces'', see Question~\ref{Ques::StrBana}. 
\end{rmk}

\begin{rmk}[Application to fractional Sobolev spaces]\label{Rmk::WspHsp}
    In literature there are two standard fractional Sobolev spaces, the Sobolev-Bessel spaces $H^{s,p}=\Fs_{p2}^s$ and the Sobolev–Slobodeckij spaces $W^{s,p}=\Bs_{pp}^s$ for $1<p<\infty$ and $s\in\R_+\backslash\Z$. See \cite[Page~34]{TriebelTheoryOfFunctionSpacesI} for a short description.
    
    These two spaces are different when $p\neq2$. As a result $T$ is bounded in Sobolev–Slobodeckij spaces but not in Sobolev-Bessel space (unless $p=2$). 
\end{rmk}
\begin{rmk}[$T$ is not H\"ormander-Mikhlin multiplier]We say that $m(\xi):\R^n\to\C$ is a H\"ormander-Mikhlin multiplier if $\sup_{j\in\Z}\|m(2^{-j}\xi)\|_{H^{s}(\frac12<|\xi|<2)}<\infty$ for some $s>\frac n2$. The multiplier theorem shows that for such $m$ the operator $[f\mapsto(m\hat f)^\vee]:L^p(\R^n)\to L^p(\R^n)$ is bounded for all $1<p<\infty$. 

The operator $T$ is indeed a convolution operator, in other words a Fourier multiplier operator. But $T$ is not bounded in $L^p=\Fs_{p2}^0$ for  $p\in(1,\infty)\backslash\{2\}$, as a result it is not a H\"ormander-Mikhlin multiplier. In fact for its multiplier $m(\xi)$ (see \eqref{Eqn::FourierMultiplier}) $\sup_{j\in\Z}\|m(2^{-j}\xi)\|_{H^{s}(\frac12<|\xi|<2)}=\infty$ for all $s>0$. This follows from the fact that $\|e^{-2\pi i y_j\cdot\xi}\|_{H^{s}(\frac12<|\xi|<2)}\xrightarrow{j\to\infty}\infty$ as we have $y_j\to\infty$ (see also \eqref{Eqn::NablaM}).
    
\end{rmk}

\begin{rmk}[On homogeneous function spaces]
    Note that the image of $Tf$ has Fourier support away from the origin. Therefore the same result is true if we replace $\Bs_{pq}^s$ and $\Fs_{pq}^s$ by the homogeneous spaces $\dot\Bs_{pq}^s$ and $\dot\Fs_{pq}^s$ respectively. We leave the details to the reader.
\end{rmk}

\begin{rmk}
    No matter how rapidly $|y_j|$ grows, $T$ is always defined on Besov functions. The assumption $|y_j|\le 2^{N_0j}$ is only used to ensure that $T$ is defined on space of tempered distributions.  As a corollary we get an alternative proof that $\Ss'(\R^n)\backslash\bigcup_{p,q,s}\Bs_{pq}^s(\R^n)\neq\varnothing$, i.e. not every tempered distributions are Besov functions.
\end{rmk}

Here by a \textit{Littlewood-Paley family} we mean a sequence of Schwartz functions $\phi=(\phi_0,\phi_1,\dots)\in\Ss(\R^n)$ such that their Fourier transform $\hat\phi_j(\xi)=\int \phi_j(x)e^{-2\pi ix\xi}dx$ satisfy
\begin{itemize}
    \item $\supp\hat\phi_0\subset B(0,2)$ and $\hat\phi_0\equiv1$ in a neighborhood of $\overline {B(0,1)}$.
    \item $\hat\phi_j(\xi)=\hat\phi_0(2^{-j}\xi)-\hat\phi_0(2^{1-j}\xi)$ for $j\ge1$. 
\end{itemize}

As a result $\supp\hat\phi_j\subset\{2^{j-1}<|\xi|<2^{j+1}\}$ for all $j\ge1$.

For $0<p,q\le\infty$ and $s\in\R$ the \textit{Besov and Triebel-Lizorkin norms associated to} $\phi$ are defined by
\begin{align}
\label{Eqn::BsNorm}
\|f\|_{\Bs_{pq}^s(\phi)}&:=\|(2^{js}\phi_j\ast f)_{j=0}^\infty\|_{\ell^q(\N_{\ge0};L^p(\R^n))}=\bigg(\sum_{j=0}^\infty2^{jsq}\Big(\int_{\R^n}|\phi_j\ast f(x)|^pdx\Big)^\frac qp\bigg)^\frac1q;
\\
    \label{Eqn::TLNorm1}
    \|f\|_{\Fs_{pq}^s(\phi)}&:=\|(2^{js}\phi_j\ast f)_{j=0}^\infty\|_{L^p(\R^n;\ell^q(\N_{\ge0}))}=\bigg(\int_{\R^n}\Big(\sum_{j=0}^\infty|2^{js}\phi_j\ast f(x)|^q\Big)^\frac pqdx\bigg)^\frac1p,&p<\infty;
    \\
    \label{Eqn::TLNorm2}
    \|f\|_{\Fs_{\infty q}^s(\phi)}&:=\sup_{x\in\R^n,J\in\Z}2^{J\frac nq}\|(2^{js}\phi_j\ast f)_{j=\max(J,0)}^\infty\|_{L^q(B(x,2^{-J});\ell^q)},&p=\infty.
\end{align}

For $\As\in\{\Bs,\Fs\}$ we define $\As_{pq}^s(\R^n)=\{f\in\Ss'(\R^n):\|f\|_{\As_{pq}^s(\phi)}<\infty\}$ by a fixed choice of $\phi$. Different choice $\phi$ results in equivalent norm (see e.g. \cite[Proposition~2.3.2]{TriebelTheoryOfFunctionSpacesI} and \cite[Propositions 1.3 and 1.8]{TriebelTheoryOfFunctionSpacesIV}).

\medskip
In the following $\1_U:\R^n\to\{0,1\}$ denotes the characterization of $U\subset\R^n$. We use $\1=\1_{\R^n}$.
We use the notation $A \lesssim B$ to denote that $A \leq CB$, where $C$ is a constant independent of $A,B$. We use $A \approx B$ for ``$A \lesssim B$ and $B \lesssim A$''. And we use $A\lesssim_pB$ to emphasize that the constant depends on the quantity $p$.

\section{Proof of the Theorem}

The boundedness of $T$ in $\Ss'$ uses the characterization of multipliers on Schwartz space, originally given in \cite{SchwartzBook}. See also \cite{LarcherMultiplier}.
\begin{proof}[Proof of Theorem~\ref{Thm::TBdd}~\ref{Item::TBdd::Ss}]
    Applying Fourier transform we have, for every Schwartz function $f$,
    \begin{equation}\label{Eqn::FourierMultiplier}
        (Tf)^\wedge(\xi)=\sum_{j=1}^\infty e^{-2\pi iy_j\xi}\hat\phi_j(\xi)\hat f(\xi)=\sum_{j=1}^\infty e^{-2\pi iy_j\xi}\hat\phi_1(2^{1-j}\xi)\hat f(\xi)=:m(\xi)\hat f(\xi).
    \end{equation}

    Since Fourier transform is isomorphism on space of tempered distributions, $T:\Ss'(\R^n)\to\Ss'(\R^n)$ is bounded if and only if the multiplier operator $[g\mapsto mg]:\Ss'(\R^n)\to\Ss'(\R^n)$ is bounded. By taking adjoint this holds if and only if $[g\mapsto mg]:\Ss(\R^n)\to\Ss(\R^n)$ is bounded.

    By characterization of Schwartz multipliers (see for example \cite[Proposition~4.11.5, page~417]{Horvath}), $[g\mapsto mg]:\Ss(\R^n)\to\Ss(\R^n)$ is bounded if and only if
    \begin{equation}\label{Eqn::Multiplier}
         m\in C^\infty_\loc(\R^n)\text{ and } \forall k\ge0,\ \exists M_k>1,\  \forall\xi\in\R^n,\quad |\nabla^k m(\xi)|\le M_k(1+|\xi|)^{M_k}.
    \end{equation}

    If there is $N_0$ such that $|y_j|\le 2^{N_0j}$ for all $j$, then for every $j\ge1$ and $2^{j-1}<|\xi|<2^{j+1}$, $$|\nabla^k(e^{-2\pi iy_j\xi}\hat\phi_1(2^{1-j}\xi))|\lesssim_k|\nabla^{\le k} e^{-2\pi iy_j\xi}|\cdot|\nabla^{\le k}(\hat\phi_1(2^{1-j}\xi))|\lesssim |2\pi y_j|^k\sum_{l=0}^k2^{(1-j)l}|\nabla^l\hat\phi_1|\lesssim_{\phi,k}2^{N_0kj}. $$

    That is to say there is a $C_k>1$ such that $|\nabla^km(\xi)|\le C_k(1+|\xi|)^{N_0k}$ for all $\xi$. Taking $M_k=\max(C_k,N_0k)$, \eqref{Eqn::Multiplier} is satisfied and hence $m$ is a Schwartz multiplier.

    Conversely, using $\supp\nabla\hat\phi_1\subset\{1<|\xi|<2\}\cup\{2<|\xi|<4\}$ and $\hat\phi_1(2)=1$, for every $|\xi_0|=1$ we have 
    \begin{equation}\label{Eqn::NablaM}
        (\nabla m)(2^j\xi_0)=\nabla_\xi(e^{-2\pi iy_j\xi})|_{\xi=2^j\xi_0}=e^{-2\pi i2^jy_j\xi_0}\cdot(-2\pi iy_j)\quad\Rightarrow\quad|(\nabla m)(2^j\xi_0)|=2\pi|y_j|.
    \end{equation}

    Therefore if $m$ is a Schwartz multiplier, taking $k=1$ in \eqref{Eqn::Multiplier} we get $2\pi|y_j|\le (1+2^j)^{M_1}\le 2^{j(M_1+1)}$ for all $j\ge1$. Taking $N_0=M_1+1$ we get $|y_j|\le 2^{N_0j}$ for all $j\ge1$.
\end{proof}

The boundedness in Besov spaces follows from direct computations.

\begin{proof}[Proof of Theorem~\ref{Thm::TBdd}~\ref{Item::TBdd::Bs}]
    The support assumption of $\hat\phi$ gives $\hat\phi_j=(\hat\phi_{j-1}+\hat\phi_j+\hat\phi_{j+1})\hat\phi_j$ for all $j\ge0$ (here we use $\phi_{-1}=0$).  
    
    The standard estimate yields $\|\phi_j\ast\phi_k\ast f\|_{L^p}\lesssim_p\|\phi_k\ast f\|_{L^p}$ for $|k|\le 1$, see e.g. \cite[(2.3.2/4)]{TriebelTheoryOfFunctionSpacesI}, which can be done via either H\"ormander-Mikhlin multipliers or Peetre's maximal functions.

    Therefore $\phi_j\ast Tf=\phi_j\ast\sum_{k=j-1}^{j+1}\tau_{y_k}(\phi_k\ast f)=\sum_{k=j-1}^{j+1}\tau_{y_k}(\phi_j\ast \phi_k\ast f)$, which means    
    \begin{equation}\label{Eqn::Bsbdd::MainIneqn}
        \begin{aligned}
        \|Tf\|_{\Bs_{pq}^s(\phi)}=&\Big\|\Big(2^{js}\sum_{k=j-1}^{j+1}\tau_{y_k}(\phi_j\ast \phi_k\ast f)\Big)_{j=0}^\infty\Big\|_{\ell^q(L^p)}\lesssim_{p,q}\sum_{k=j-1}^{j+1}\big\|\big(2^{js}\tau_{y_k}(\phi_j\ast\phi_k\ast f)\big)_{j=0}^\infty\big\|_{\ell^q(L^p)}
        \\
        =&\sum_{k=j-1}^{j+1}\big\|\big(2^{js}\|\phi_j\ast\phi_k\ast f\|_{L^p}\big)_{j=0}^\infty\big\|_{\ell^q}\lesssim_p\sum_{k=j-1}^{j+1}\big\|\big(2^{js}\|\phi_k\ast f\|_{L^p}\big)_{j=0}^\infty\big\|_{\ell^q}
        \\
        \lesssim&_{s}\sum_{k=j-1}^{j+1}\big\|\big(2^{ks}\phi_k\ast f\big)_{j=0}^\infty\big\|_{\ell^q(L^p)}\approx\|f\|_{\Bs_{pq}^s(\phi)}.
    \end{aligned}
    \end{equation}

    This proves $T:\Bs_{pq}^s(\R^n)\to\Bs_{pq}^s(\R^n)$ for all $0<p,q\le\infty$ and $s\in\R$.
\end{proof}

Next for each $p\neq q$ we construct examples $f=f_{pqs}\in\Fs_{pq}^s(\R^n)$ such that $Tf\notin \Fs_{pq}^s(\R^n)$. 

Let $\mu_0:=\frac12\inf_{j\neq k}|y_j-y_k|>0$. Fix a $y_0\in\R^n$ such that $|y_0|=1$. We set
\begin{align}\label{Eqn::DefChi}
    &\chi\in C_c^\infty(B(0,2\mu_0))\text{ such that }\1_{B(0,\mu_0)}\le\chi\le\1_{B(0,2\mu_0)};
    \\\label{Eqn::DefExp}
    &e_j(x):=\exp(2\pi i 2^jy_0\cdot x),\quad \tilde e_j(x):=e_j(-x)=\exp(-2\pi i2^jy_0\cdot x)\quad\text{for }j\ge0.
\end{align}

Notice that $e_j(x-y)=e_j(x)e_j(-y)=e_j(x)\tilde e_j(y)$. Therefore, for $j\ge0$, $g,h\in\Ss(\R^n)$ and $x\in\R^n$,
\begin{equation}\label{Eqn::ChangeConv}
    |g\ast (h e_j)(x)|=\bigg|\int g(y)h(x-y)e_j(x-y)dy\bigg|=\bigg|e_j(x)\int g(y)\tilde e_j(y)h(x-y)dy\bigg|=|(g\tilde e_j)\ast h(x)|.
\end{equation}

Our counterexample function $f=f_{pqs}$ would have the form 
$$f(x)=\sum_{j=1}^\infty a_j2^{-js}\cdot(\chi e_{j})(x+u_j)$$
\begin{itemize}
    \item when $p<q$, we require $(a_j)_j\in\ell^q\backslash\ell^p$ and $u_j\equiv0$;
    \item when $p>q$, we require $(a_j)_j\in\ell^p\backslash\ell^q$ and $u_j\equiv y_{j}$.
\end{itemize}

Either case we want $2^{js}\phi_{j}\ast f\approx a_j\cdot\tau_{-u_j}(\chi e_{j})\approx a_j\cdot e_j\1_{B(-u_j,\mu_0)}$. More precisely
\begin{prop}\label{Prop::DecayPhiE}
    For every $M\ge1$ there is $C=C(M,\mu_0,\phi,\chi)>0$ such that,
    \begin{equation}\label{Eqn::DecayPhiE1}
        |\phi_j\ast(\chi e_k)(x)|\le C 2^{-M\max(j,k)}(1+2^j\max(0,|x|-2\mu_0))^{-M},\quad\text{for every }0\le j\neq k\quad\text{and}\quad x\in\R^n.
    \end{equation}
    In particular there is a $C'=C'(M,\mu_0,\phi,\chi)>0$ such that
    \begin{equation}\label{Eqn::DecayPhiE2}
        |\phi_j\ast(\chi e_k)(x)|\le C' 2^{-M|j-k|}(1+|x|)^{-M},\quad\text{for all }j,k\ge0\text{ and }x\in\R^n.
    \end{equation}
\end{prop}
\begin{proof}
 By assumption there is a $\epsilon_0>0$ such that $\supp\hat\phi_0\subset\{|\xi|<2^{1-\epsilon_0}\}$ and $\hat\phi_0|_{B(0,2^{\epsilon_0})}\equiv1$. Therefore $\supp\hat\phi_j\subset\{2^{j-1+\epsilon_0}<|\xi|<2^{j+1-\epsilon_0}\}$ for all $j\ge1$. Take $\rho_0=\rho_0(\eps_0)\ge1$ such that $1-2^{-\epsilon_0}\ge 2^{2-\rho_0}$. In particular $2^{\epsilon_0}-1\ge 2^{2-\rho_0}$ as well.
Note that $\supp(\tilde e_k)^\wedge=\{-2^ky_0\}\subset\{|\xi|=2^k\}$ for all $k\ge0$.  Therefore,
    \begin{equation}\label{Eqn::DecayPhiE::Supp}
        \supp(\phi_j\tilde e_k)^\wedge\subset\{2^{\max(j,k)-\rho_0+1}<|\xi|<2^{\max(j,k)+\rho_0-1}\},\quad\text{for all }j,k\ge0\quad\text{such that}\quad k\neq j.
    \end{equation}

    Let us define $(\psi_l)_{l\in\Z}\subset\Ss(\R^n)$ by $\hat\psi_l(\xi):=\hat\phi_1(2^{1-l}\xi)$. Therefore $\supp\hat\psi_l\subset\{2^{l-1}<|\xi|<2^{l+1}\}$ for all $l\in\Z$ and $\sum_{l\in\Z}\hat\psi_l(\xi)=1$ for $\xi\neq0$.  
We conclude that,
$$\phi_j\tilde e_k=\sum_{l=\max(j,k)-\rho_0}^{\max(j,k)+\rho_0}\psi_{l}\ast(\phi_j\tilde e_k),\quad\text{for all } j,k\ge0\text{ such that } j\neq k.$$

    Let us assume $M$ to be even without loss of generality. Since $\psi_{l}$ has Fourier support away from 0, we have $\psi_l\ast\chi=(\Delta^{-\frac M2}\psi_l)\ast(\Delta^\frac M2\chi)$ with $\Delta^\frac M2\chi$ still supported in $\supp\chi\subset B(0,2\mu_0)$, which means
    \begin{align*}
        &|\phi_j\ast(\chi e_k)(x)|\overset{\eqref{Eqn::ChangeConv}}=|(\phi_j\tilde e_k)\ast\chi(x)|\le\sum_{l=\max(j,k)-\rho_0}^{\max(j,k)+\rho_0}|(\phi_j\tilde e_k)\ast\psi_l\ast\chi(x)|
        \\
        =&\sum_{l=\max(j,k)-\rho_0}^{\max(j,k)+\rho_0}|(\phi_j\tilde e_k)\ast\Delta^{-\frac M2}\psi_l\ast\Delta^{\frac M2}\chi(x)|\le\sum_{l=\max(j,k)-\rho_0}^{\max(j,k)+\rho_0}|\phi_j|\ast|\Delta^{-\frac M2}\phi_l|\ast|\Delta^\frac M2\chi|(x)
        \\
        \le&\|\chi\|_{C^M}\sum_{l=\max(j,k)-\rho_0}^{\max(j,k)+\rho_0}|\phi_j|\ast|\Delta^{-\frac M2}\phi_l|\ast\1_{B(0,2\mu_0)}(x)\lesssim_\chi\sum_{l=\max(j,k)-\rho_0}^{\max(j,k)+\rho_0}\int_{B(0,2\mu_0)}|\phi_j|\ast|\Delta^{-\frac M2}\psi_l|(x-y)dy
        \\
        \le&\sum_{l=\max(j,k)-\rho_0}^{\max(j,k)+\rho_0} \iint_{|s|+|t|\ge\max(0,|x|-2\mu_0)}|\phi_j(t)||\Delta^{-\frac M2}\psi_l(s)|dtds
        \\
        \le&\sum_{l=\max(j,k)-\rho_0}^{\max(j,k)+\rho_0}\bigg(\|\phi_j\|_{L^1}\int_{|s|\ge\max(0,\frac12|x|-\mu_0)}|\Delta^{-\frac M2}\psi_l(s)|ds+\|\Delta^{-\frac M2}\psi_l\|_{L^1}\int_{|t|\ge\max(0,\frac12|x|-\mu_0)}|\phi_j(t)|dt\bigg)
        \\
        \lesssim&_{\phi,M}\sum_{l=\max(j,k)-\rho_0}^{\max(j,k)+\rho_0}2^{-Ml}\bigg(\int_{|s|\ge 2^l\max(0,\frac12|x|-\mu_0)}\hspace{-0.2in}|\Delta^{-\frac M2}\psi_0(s)|ds+\int_{|t|\ge 2^{j-1}\max(0,\frac12|x|-\mu_0)}\hspace{-0.2in}(|\phi_0(t)|+|\phi_1(t)|)dt\bigg)
        \\
        \lesssim&_{M,\phi}\sum_{l=\max(j,k)-\rho_0}^{\max(j,k)+\rho_0}\big(2^{-Ml}(1+2^l\max(0,\tfrac12|x|-\mu_0))^{-M}+2^{-Ml}(1+2^{j-1}\max(0,\tfrac12|x|-\mu_0))^{-M}\big)
        \\
        \lesssim&_{\mu_0} 2^{-M\max(j,k)}(1+2^{j-\rho_0}\max(0,\tfrac12|x|-\mu_0))^{-M}\lesssim_{\phi}2^{-M\max(j,k)}(1+2^j\max(0,|x|-2\mu_0))^{-M}.
    \end{align*}

    Therefore \eqref{Eqn::DecayPhiE1} holds for all $j\neq k$.

    For \eqref{Eqn::DecayPhiE2}, when $j\neq k$, \eqref{Eqn::DecayPhiE2} follows from 
\eqref{Eqn::DecayPhiE1} with $2^{\max(j,k)}(1+2^j\max(0,|x|-2\mu_0))\gtrsim_{\mu_0}2^{|j-k|}(1+|x|)$. When $j=k$, \eqref{Eqn::DecayPhiE2} is obtained from the following decay estimates: when $|x|\ge 4\mu_0$,
    \begin{align*}
    |\phi_j\ast(\chi e_j)(x)|\le&\int_{B(0,2\mu_0)}|\phi_j(x-y)|dy\le\int_{|y|>|x|/2}|\phi_j(y)|dy
    \\
    \le&\int_{|y|>2^{j-1}|x|}(|\phi_0(y)|+2^n|\phi_1(2y)|)dy\lesssim_{M,\mu_0}(1+|x|)^{-M}.\qedhere
\end{align*}
\end{proof}

For $p>q$ we want the estimate $\|f_{pqs}\|_{\Fs_{pq}^s}\lesssim\|(a_j)_j\|_{\ell^p}$, which is obtained from the following:

\begin{lem}\label{Lem::DiscreteSum}
    Let $\varphi:\R^n\to\R_+$ be a positive bounded function such that $\sup_x(1+|x|)^{n+1}|\varphi(x)|<\infty$. Let $(y_j)_j$ be from the assumption that $\inf_{j\neq k}|y_j-y_k|\ge2\mu_0$. Then for every $1\le r\le\infty$ there is a $C=C(r,\varphi)>0$ such that for every  $b=(b_j)_{j=1}^\infty$,
    \begin{gather}
        \label{Eqn::DiscreteSum1}
        \bigg\|\sum_{j=1}^\infty b_j\tau_{y_j}\varphi\bigg\|_{L^r(\R^n)}\le C\|b\|_{\ell^r};
        \\
        \label{Eqn::DiscreteSum2}
        \sup_{R>0;x\in\R^n}R^{n/r}\bigg\|\sum_{j=1}^\infty b_j\tau_{y_j}\varphi\bigg\|_{L^r(B(x,R))}\le C\|b\|_{\ell^\infty}.
    \end{gather}
\end{lem}
The result holds for $r<1$ if $\varphi$ has a faster decay. In application we will use $r= p/q$ where $q<p$.
\begin{proof}
Note that for every $g\in L^\infty(\R^n)$, $R>0$ and $x\in\R^n$ we have 
\begin{equation}\label{Eqn::DiscreteSumLInfty}
    R^\frac nr\|g\|_{L^r(B(x,R))}\le|B(0,1)|^\frac1r\|g\|_{L^\infty}.
\end{equation} Therefore \eqref{Eqn::DiscreteSum2} is implied by taking $r=\infty$ in \eqref{Eqn::DiscreteSum1}.

    Let $\tilde\varphi(x)=\sup_{|y|<\mu_0}|\varphi(x+y)|$. Clearly $\sup_x(1+|x|)^{n+1}\tilde \varphi(x)<\infty$, thus $\tilde\varphi$ is still integrable. Therefore $\varphi(x)\le|B(0,\mu_0)|^{-1}\1_{B(0,\mu_0)}\ast\tilde\varphi(x)$ which means
    \begin{equation*}
        \bigg\|\sum_{j=1}^\infty b_j\tau_{y_j}\varphi\bigg\|_{L^r}=\bigg\|\sum_{j=1}^\infty b_j(\delta_{y_j}\ast\varphi)\bigg\|_{L^r}\le\bigg\|\sum_{j=1}^\infty b_j\frac{\1_{B(y_j,\mu_0)}}{|B(0,\mu_0)|}\ast\tilde\varphi\bigg\|_{L^r}\le\frac{\|\tilde\varphi\|_{L^1}}{|B(0,\mu_0)|}\bigg\|\sum_{j=1}^\infty b_j\1_{B(y_j,\mu_0)}\bigg\|_{L^r}.
    \end{equation*}
    
    Since $(B(y_j,\mu_0))_{j=1}^\infty$ are all disjointed, we get $\|\sum_{j=1}^\infty b_j\1_{B(y_j,\mu_0)}\|_{L^r}=\|b\|_{\ell^r}\|\1_{B(0,\mu_0)}\|_{L^r}$, finishing the proof of \eqref{Eqn::DiscreteSum1} and hence the whole lemma.
\end{proof}

Next we bound $\|Tf\|_{\Fs_{pq}^s}$ from below.
Recall from the assumption and construction that $(y_j)_{j=1}^\infty$ satisfy $\inf_{j\neq k}|y_j-y_k|\ge2\mu_0$ and $\chi$ satisfies $\1_{B(0,\mu_0)}\le\chi\le\1_{B(0,2\mu_0)}$.
\begin{prop}\label{Prop::BddBelow}
    For every $N>1$ there is a $K=K(N,\phi,\mu_0)\ge1$ such that for every $(u_k)_{k=1}^\infty\subset\R^n$, $j\ge K$ and $x\in B(y_j-u_j,\frac12\mu_0)$,
    \begin{equation}\label{Eqn::BddBelow1}
        |\phi_j\ast\phi_j\ast\tau_{y_j-u_j}(\chi e_j)(x)|-\sum_{\substack{k,l=1\\(k,l)\neq(j,j)}}^\infty2^{N|l-j|}|\phi_j\ast\phi_k\ast\tau_{y_k-u_l}(\chi e_l)(x)|\ge\tfrac12.
    \end{equation}

    In particular let $(a_j)_{j=1}^\infty\subset\C$ be such that $|a_j|\le2^{|j-k|}|a_k|$ for all $j,k\ge 1$, then for every $|s|\le N-1$, 
    \begin{equation}\label{Eqn::BddBelow2}
        2^{js}\bigg|\phi_j\ast\sum_{k,l=1}^\infty2^{-ls}a_l\cdot\tau_{y_k-u_l}\big(\phi_k\ast(\chi e_l)\big)(x)\bigg|\ge\frac{|a_j|}2,\quad\text{for }j\ge K\text{ and } x\in B(y_j-u_j,\tfrac12\mu_0). 
    \end{equation}
\end{prop}
\begin{proof}
    Recall that from \eqref{Eqn::ChangeConv} that $|\phi_j\ast\phi_j\ast(e_j\chi)|=|((\phi_j\ast\phi_j)\tilde e_j)\ast \chi|$. 
    
    Note that $\int(\phi_j\ast\phi_j)\cdot\tilde e_j=(\phi_j\ast\phi_j)^\wedge(-2^jy_0)=\hat\phi_j(-2^jy_0)^2=1$. We see that $((\phi_j\ast\phi_j)\cdot\tilde e_j)\ast\1(x)=1$ for all $x\in\R^n$.

    Since $\phi_1$ rapidly decay we have for every $j\ge1$
    $$\int_{|y|>\frac12\mu_0}|(\phi_j\ast\phi_j)\cdot\tilde e_j(y)|dy=\int_{|y|>\frac12\mu_0}|\phi_j\ast\phi_j(y)|dy\le\int_{|y|>2^{2-j}\mu_0}|\phi_1\ast\phi_1(y)|dy\lesssim_{M,\mu_0}2^{-Mj}.$$
    In particular there is a $C_1=C_1(\phi,\mu_0)>0$ such that $\int_{|y|>\frac12\mu_0}|(\phi_j\ast\phi_j)\tilde e_j|\le C_12^{-j}$.
    
    Recall $\chi|_{B(0,\mu_0)}\equiv1$ from \eqref{Eqn::DefChi}. Therefore for $|x|<\frac12\mu_0$ and $j\ge1$,
    \begin{align*}
        &|\phi_j\ast\phi_j\ast(e_j\chi)(x)|=|((\phi_j\ast\phi_j)\cdot\tilde e_j)\ast\chi(x)|\ge|((\phi_j\ast\phi_j)\cdot\tilde e_j)\ast\1(x)|-|(\phi_j\ast\phi_j)\cdot\tilde e_j|\ast(\1-\chi)(x)
        \\
        \ge&1-\int_{|x-y|>\mu_0}|(\phi_j\ast\phi_j)\cdot\tilde e_j(y)|(1-\chi(x-y))dy\ge1-\int_{|y|>\frac12\mu_0} |\phi_j\ast\phi_j(y)|dy\ge 1-C_12^{-j}.
    \end{align*}
    
    By taking translation, this is to say
\begin{equation}\label{Eqn::BddBelow::Tmp1}
    |\phi_j\ast\phi_j\ast\tau_{y_j-u_j}(e_j\chi)|\ge(1-C_12^{-j})\cdot\1_{B(y_j-u_j,\frac12\mu_0)},\quad \text{for all }j\ge1.
\end{equation}   

    On the other hand since $\phi_j\ast\phi_k=0$ for $|j-k|\ge2$, we can assume the index $k$ in \eqref{Eqn::BddBelow1} satisfies $|k-j|\le1$. When $(k,l)\neq(j,j)$, by \eqref{Eqn::DecayPhiE1}, 
    \begin{equation*}
        \|\phi_j\ast\phi_k\ast(\chi e_l)\|_{L^\infty(\R^n)}\lesssim_{N,\mu_0,\phi,\chi}
        \begin{cases}
        \|\phi_j\|_{L^1}2^{-(N+2)\max(k,l)}&k\neq l\\\|\phi_k\|_{L^1}2^{-(N+2)\max(j,l)}&j\neq l    
        \end{cases}\approx_{N,\phi}2^{-(N+2)\max(j,l)}.
    \end{equation*}

    Therefore there is a $C_2>0$ such that 
    \begin{equation}\label{Eqn::BddBelow::Tmp2}
        \|\phi_j\ast\phi_k\ast(\chi e_l)\|_{L^\infty}\le C_22^{-(N+1)\max(j,l)},\quad \text{for all }j,k,l\ge1\text{ such that }(k,l)\neq(j,j).
    \end{equation}
    
    Combining \eqref{Eqn::BddBelow::Tmp1} and \eqref{Eqn::BddBelow::Tmp2} we have for every $j\ge 1$ and $x\in\R^n$,
    \begin{align*}
        &|\phi_j\ast\phi_j\ast\tau_{y_j-u_j}(\chi e_j)(x)|-\sum_{k,l\ge1;(k,l)\neq(j,j)}2^{N|l-j|}|\phi_j\ast\phi_k\ast\tau_{y_k-u_l}(\chi e_l)(x)|
        \\
        \ge&(1-C_12^{-j})\cdot\1_{B(y_j-u_j,\frac12\mu_0)}(x)-\sum_{k=j-1}^{j+1}\sum_{l=1}^\infty 2^{N|l-j|}C_22^{-(N+1)\max(j,l)}
        \\
        \ge&(1-C_12^{-j})\cdot\1_{B(y_j-u_j,\frac12\mu_0)}(x)-3C_2\bigg(\sum_{l=1}^j2^{N(j-l)-(N+1)j}+\sum_{l=j+1}^\infty 2^{N(l-j)-(N+1)l}\bigg)
        \\
        \ge&(1-C_12^{-j})\cdot\1_{B(y_j-u_j,\frac12\mu_0)}(x)-6C_22^{-j}.
    \end{align*}

    Take $K$ such that $(C_1+6C_2)2^{-K}\le\frac12$, i.e.  $K\ge1+\log_2(C_1+6C_2)$, we get \eqref{Eqn::BddBelow1}.

    Suppose $|s|\le N-1$ and $|a_j|\le2^{|j-k|}|a_k|$ holds for all $j,k\ge1$. We see that for every $x$ and $j\ge 1$,
    \begin{align*}
        \bigg|\sum_{\substack{k,l\ge1\\(k,l)\neq(j,j)}}2^{(j-l)s}a_l\cdot\phi_j\ast\phi_k\ast\tau_{y_k-u_l}(\chi e_l)(x)\bigg|
        \le &\sum_{\substack{k,l\ge1\\(k,l)\neq(j,j)}}2^{|j-l|(|s|+1)}|a_j|\cdot|\phi_j\ast\phi_k\ast\tau_{y_k-u_l}(\chi e_l)(x)|
        \\
        \le&|a_j|\sum_{\substack{k,l\ge1\\(k,l)\neq(j,j)}}2^{|j-l|N}\cdot|\phi_j\ast\phi_k\ast\tau_{y_k-u_l}(\chi e_l)(x)|.
    \end{align*}

    Applying \eqref{Eqn::BddBelow1} for $j\ge K$ we get \eqref{Eqn::BddBelow2} immediately.
\end{proof}

We now prove Theorem~\ref{Thm::TBdd}~\ref{Item::TBdd::Fs}. Recall the following convolution inequality, see e.g. \cite[Lemma~2]{RychkovThmofBui}: for every $0<p,q\le\infty$ and $\delta>0$ there is a $C_{p,q,\delta}>0$ such that
\begin{equation}\label{Eqn::FsConvIne}
    \Big\|\Big(\sum_{j=1}^\infty2^{-\delta|j-k|}g_j\Big)_{k=0}^\infty\Big\|_{L^p(\R^n;\ell^q)}\le C_{p,q,\delta}\|(g_j)_{j=1}^\infty\|_{L^p(\R^n;\ell^q)},\qquad g=(g_j)_{j=1}^\infty:\R^n\to\ell^q(\Z_+).
\end{equation}
Notice that if $p,q\ge1$ this follows directly from Young's convolution inequality on $\Z$.

\begin{proof}[Proof of Theorem~\ref{Thm::TBdd}~\ref{Item::TBdd::Fs}]
    Let $\chi$ be from \eqref{Eqn::DefChi}, $(e_j)_{j=1}^\infty$ be from \eqref{Eqn::DefExp}.
    
    Let $(u_j)_{j=1}^\infty\subset\R^n$ and $(a_j)_{j=1}^\infty\in\ell^\infty$ to be determined later, such that $|a_j|\le 2^{|j-k|}|a_k|$ for all $j,k\ge1$. For $s\in\R$ we define \begin{equation}\label{Eqn::DefF}
        f_{s,a,u}:=\sum_{j=1}^\infty 2^{-js}a_j\cdot\tau_{-u_j}(\chi e_{j}).
    \end{equation}
    
    Therefore when $p<\infty$,
    \begin{align}
    \notag
        \|f_{s,a,u}\|_{\Fs_{pq}^s(\phi)}=&\Big(\int_{\R^n}\Big\|\Big(\sum_{k=1}^\infty2^{(j-k)s}a_k\cdot\tau_{-u_k}\phi_j\ast(\chi e_{k})(y)\Big)_{j=0}^\infty\Big\|_{\ell^q}^p dy\Big)^{1/p}
        \\
    \notag
        \lesssim&_M\Big(\int_{\R^n}\Big\|\Big(\sum_{k=1}^\infty\frac{2^{(j-k)s-|j-k|M}a_k}{(1+|y+u_k|)^M}\Big)_{j=0}^\infty\Big\|_{\ell^q}^p dy\Big)^{1/p}&\text{(by \eqref{Eqn::DecayPhiE2})}
        \\
    \label{Eqn::FsBddUpforF1}
        \lesssim&_{p,q}\big\|\big(a_k(1+|y+u_k|)^{-M}\big)_{k=1}^\infty\big\|_{L^p_y(\R^n;\ell^q)}&\hspace{-0.2in}\text{(by \eqref{Eqn::FsConvIne} with $M\ge |s|+1$)}.
    \end{align}
When $p=\infty$, similarly by Proposition~\ref{Prop::DecayPhiE} and \eqref{Eqn::FsConvIne} we have, for $M\ge|s|+1$,
    \begin{align}
        \notag\|f_{s,a,u}\|_{\Fs_{\infty q}^s(\phi)}=&\sup_{x\in\R^n,J\in\Z}2^{J\frac nq}\Big(\int_{B(x,2^{-J})}\Big\|\Big(\sum_{k=1}^\infty2^{(j-k)s}a_k\cdot\tau_{-u_k}\phi_j\ast(\chi e_{k})(y)\Big)_{j=\max(J,0)}^\infty\Big\|_{\ell^q}^q dy\Big)^{1/q}
        \\
    \notag
        \lesssim&_M\sup_{x\in\R^n,J\in\Z}2^{J\frac nq}\Big(\int_{B(x,2^{-J})}\Big\|\Big(\sum_{k=1}^\infty\frac{2^{(j-k)s-|j-k|M}a_k}{(1+|y+u_k|)^M}\Big)_{j=\max(J,0)}^\infty\Big\|_{\ell^q}^q dy\Big)^{1/q}
        \\
    \label{Eqn::FsBddUpforF2}
        \lesssim&_{q}\sup_{x\in\R^n,J\in\Z}2^{J\frac nq}\big\|\big(a_k(1+|y+u_k|)^{-M}\big)_{k=1}^\infty\big\|_{L^q_y(B(x,2^{-J});\ell^q)}.
    \end{align}

    Recall that by \eqref{Eqn::DefT} and \eqref{Eqn::DefF},
    \begin{equation*}
        \phi_j\ast T(f_{s,a,u})=\phi_j\ast\sum_{k=1}^\infty \tau_{y_k}(\phi_k\ast f_{s,a,u})=\phi_j\ast\sum_{k,l=1}^\infty 2^{-ls}a_l\cdot\tau_{y_k-u_l}\big(\phi_k\ast(\chi e_l)\big).
    \end{equation*}
    
    Now we take  $K=K(|s|+1,\phi,\mu_0)\ge1$ to be the index in Proposition~\ref{Prop::BddBelow}. Since $|a_j|\le 2^{|j-k|}|a_k|$, applying \eqref{Eqn::BddBelow2} we see that, when $p<\infty$,
    \begin{align}
        \label{Eqn::FsBddBelowforF1}\|T(f_{s,a,u})\|_{\Fs_{pq}^s(\phi)}\ge&\big\|\big(2^{js}\phi_{j}\ast T(f_{s,a,u})\big)_{j=K}^\infty\big\|_{L^p(\R^n;\ell^q)}\ge\tfrac12\big\|\big(a_j\cdot\1_{B(y_{j}-u_j,\frac12\mu_0)}\big)\big)_{j=K}^\infty\big\|_{L^p(\R^n;\ell^q)}.
    \end{align}
    When $p=\infty$, similarly we have
    \begin{align}
        \notag\|T(f_{s,a,u})\|_{\Fs_{\infty q}^s(\phi)}\ge&\sup_{x\in\R^n;J\in\Z}2^{J\frac nq}\big\|\big(2^{js}\phi_{j}\ast T(f_{s,a,u})\big)_{j=\max(J,K)}^\infty\big\|_{L^q(B(x,2^{-J});\ell^q)}
        \\
        \notag
        \ge&\tfrac12\sup_{x\in\R^n;J\in\Z}2^{J\frac nq}\big\|\big(a_j\cdot\1_{B(y_{j}-u_j,\frac12\mu_0)}\big)\big)_{j=\max(J,K)}^\infty\big\|_{L^q(B(x,2^{-J});\ell^q)}
        \\
        \label{Eqn::FsBddBelowforF2}
        \ge&\tfrac12\big\|\big(a_j\cdot\1_{B(y_{2j}-u_j,\frac12\mu_0)}\big)\big)_{j=K}^\infty\big\|_{L^q(B(0,1);\ell^q)}.
    \end{align}

    Now we separate the cases $p<q$ and $p>q$.

\medskip
    When $p<q$ we choose $u_j\equiv0$. We pick $(a_j)_{j=1}^\infty\in\ell^q\backslash\ell^p$ such that $|a_j|\le2^{|j-k|}|a_k|$  for all $j,k\ge1$, e.g. $a_j:=(j+\frac3p)^{-1/p}$. Applying \eqref{Eqn::FsBddUpforF1} with $M\ge\max(|s|,n/p)+1$,
    \begin{align*}
        \|f_{s,a,0}\|_{\Fs_{pq}^s}\lesssim\big\|(1+|x|)^{-M}\big\|_{L^p_x} \big\|(a_k)_{k=1}^\infty\big\|_{\ell^q}<\infty.
    \end{align*}

    On the other hand by \eqref{Eqn::FsBddBelowforF1} and the fact that $\big(B(y_{j},\frac12\mu_0)\big)_{j=1}^\infty$ are disjointed we have
    \begin{equation*}
        \|T(f_{s,a,0})\|_{\Fs_{pq}^s}\gtrsim\big\|\big(a_j\cdot\1_{B(y_{j},\frac12\mu_0)}\big)\big)_{j=K}^\infty\big\|_{L^p(\R^n;\ell^q)}=|B(0,\tfrac12\mu_0)|^{\frac1p}\|(a_j)_{j=K}^\infty\|_{\ell^p}\approx_{\mu_0,p}\|(a_j)_{j=K}^\infty\|_{\ell^p}=\infty.
    \end{equation*}
    We conclude that $T(f_{s,a,0})\notin\Fs_{pq}^s(\R^n)$ as desired.
    
    \medskip When $p>q$ we choose $u_j\equiv y_{j}$ for all $j\ge 1$. We pick $(a_j)_{j=1}^\infty\in\ell^p\backslash\ell^q$ such that $|a_j|\le2^{|j-k|}|a_k|$  for all $j,k\ge1$, e.g. $a_j:=(j+\frac3q)^{-1/q}$. 
    
    In this case applying \eqref{Eqn::FsBddUpforF1}, \eqref{Eqn::FsBddUpforF2} with $M\ge\max(\frac{n+1}q,|s|+1)$ and \eqref{Eqn::DiscreteSum1}, \eqref{Eqn::DiscreteSum2} with $r=\frac pq\in(1,\infty]$,
    \begin{align*}
        \|f_{s,a,y}\|_{\Fs_{pq}^s}&\lesssim\bigg(\int_{\R^n}\Big(\sum_{k=1}^\infty\frac{|a_k|^q}{(1+|x+y_{k}|)^{Mq}}\Big)^\frac pqdx\bigg)^{1/p}\lesssim\big\|(|a_k|^q)_{k=1}^\infty\big\|_{\ell^{p/q}}^{1/q}=\|(a_k)_{k=1}^\infty\|_{\ell^p},&p<\infty;
        \\
        \|f_{s,a,y}\|_{\Fs_{\infty q}^s}&\lesssim\sup_{x\in\R^n,J\in\Z}2^{J\frac nq}\bigg(\int_{B(x,2^{-J})}\sum_{k=1}^\infty\frac{|a_k|^q}{(1+|x+y_{k}|)^{Mq}}dx\bigg)^{1/q}\lesssim\|(a_k)_{k=1}^\infty\|_{\ell^\infty},&p=\infty.
    \end{align*}

    On the other hand applying \eqref{Eqn::FsBddBelowforF1} when $p<\infty$ and \eqref{Eqn::FsBddBelowforF2} when $p=\infty$, both with $u_j\equiv y_{j}$ we have
    \begin{align*}
        &\|T(f_{s,a,y})\|_{\Fs_{pq}^s}\gtrsim\|(a_j\cdot\1_{B(0,\frac12\mu_0)})_{j=K}^\infty\|_{L^p(B(0,1);\ell^q)}
        \\
        &\quad=\big|B\big(0,\max(\tfrac12\mu_0,1)\big)|^\frac1p\|(a_j)_{j=K}^\infty\|_{\ell^q}\approx_{\mu_0,p}\|(a_j)_{j=K}^\infty\|_{\ell^q}=\infty.
    \end{align*}

    We conclude that $T(f_{s,a,y})\notin\Fs_{pq}^s(\R^n)$ as desired, finishing the proof.
\end{proof}

\appendix
\section{Definition of Interpolation Spaces}\label{Section::DefInt}
\renewcommand{\thethm}{\Alph{thm}}\setcounter{thm}{0}

To include the cases $p,q<1$ for Besov and Triebel-Lizorkin spaces, we work on quasi-Banach spaces instead of Banach spaces.

A standard formulation of interpolation spaces is regarded as an image object of some interpolation functor. See also for example \cite[Chapter~2.4]{BerghLofstromInterpolation}.

Here we let $\Cf_1$ be the category of (complex) quasi-Banach spaces with morphisms being bounded linear maps. 

We let $\Cf_2$ be the category of \textit{compatible tuples} of  (complex) quasi-Banach spaces:
\begin{itemize}
    \item $\operatorname{Ob}\Cf_2$ consists of all pair of quasi-Banach spaces $(X_0,X_1)$ such that the sum space $X_0+X_1$ is a well-defined quasi-Banach space. Such $(X_0,X_1)$ is called a compatible quasi-Banach tuple.
    \item The hom set $\operatorname{Hom}_{\Cf_2}((X_0,X_1),(Y_0,Y_1))$ consists of all bounded linear map $T:X_0+X_1\to Y_0+Y_1$ such that $T|_{X_i}:X_i\to Y_i$ is bounded linear for $i=0,1$. We also call such $T$ an \textit{admissible operator} from $(X_0,X_1)$ to $(Y_0,Y_1)$.
\end{itemize} 
\begin{defn}
    An \textit{interpolation functor} is a functor $\Ff:\Cf_2\to \Cf_1$ such that 
    \begin{itemize}
        \item For every $(X_0,X_1)\in\operatorname{Ob}\Cf_2$, $X_0\cap X_1\subseteq\Ff(X_0,X_1)\subseteq X_0+X_1$, with both set inclusions being embeddings.
        \item For every $(X_0,X_1),(Y_0,Y_1)\in\operatorname{Ob}\Cf_2$ and $T\in \operatorname{Hom}_{\Cf_2}((X_0,X_1),(Y_0,Y_1))$, we have $\Ff(T)=T|_{\Ff(X_0,X_1)}$.
    \end{itemize}

\end{defn}

The classical complex interpolations $[-,-]_\theta$ and real interpolations $(-,-)_{\theta,q}$ for $0<\theta<1$, $0<q\le\infty$ are all interpolation functors. See \cite[Chapters~3 and 4]{BerghLofstromInterpolation}, also \cite[Chapter~3.11]{BerghLofstromInterpolation} for the case $0<q\le1$.

\begin{defn}\label{Defn::FunInt}
Let $\mathfrak{S}\subset\operatorname{Ob}\Cf_1$ be a collection of quasi-Banach spaces, such that $(X_0,X_1)$ are compatible tuples for all $X_0,X_1\in\mathfrak{S}$.

    We say $Y\in\operatorname{Ob}\Cf_1$ is a \textit{(categorical) interpolation space} from $\mathfrak S$, if there is an interpolation functor $\Ff:\Cf_2\to \Cf_1$ and $X_0,X_1\in\mathfrak S$ such that $Y=\Ff(X_0,X_1)$. 
\end{defn}

In this way Corollary~\ref{Cor::Interpo} can be formulated to the following:
\begin{cor}\label{Cor::InterpoReform1}
Let $0<p,q\le\infty$ and $s\in\R$ such that $p\neq q$. There are no interpolation functor $\Ff:\Cf_2\to\Cf_1$ and $0<p_0,p_1,q_0,q_1\le\infty$, $s_0,s_1\in\R$ such that $\Fs_{pq}^s(\R^n)=\Ff(\Bs_{p_0q_0}^{s_0}(\R^n),\Bs_{p_1q_1}^{s_1}(\R^n))$.
    
\end{cor}
\begin{proof}
    Suppose they exist. By assumption $\Bs_{p_0q_0}^{s_0}\cap\Bs_{p_1q_1}^{s_1}(\R^n)\subseteq\Fs_{pq}^s(\R^n)\subseteq \Bs_{p_0q_0}^{s_0}+\Bs_{p_1q_1}^{s_1}(\R^n)$.
    
    The operator $T$ in Theorem~\ref{Thm::TBdd} satisfies $T:\Bs_{p_iq_i}^{s_i}(\R^n)\to\Bs_{p_iq_i}^{s_i}(\R^n)$ for $i=0,1$. By assumption of $\Ff$, $\Ff(T):\Fs_{pq}^s(\R^n)\to \Fs_{pq}^s(\R^n)$ must be bounded. However $T|_{\Fs_{pq}^s}=\Ff(T)$ by definition, and $T(\Fs_{pq}^s(\R^n))\not\subset \Fs_{pq}^s(\R^n)$, giving a contradiction.
\end{proof}

Alternatively we can focus on local without traversing all quasi-Banach spaces. For details see e.g. \cite[Chapter~3.1]{BennettSharpleyInterpolation}.

\begin{defn}
    Let $(X_0,X_1)$ be a compatible pair of quasi-Banach spaces. We say $X$ is a \textit{(set-theoretical) interpolation space} of $(X_0,X_1)$, if
    \begin{itemize}
        \item $X_0\cap X_1\subseteq X\subseteq X_0+X_1$, both set inclusions are embeddings.
        \item For every admissible operator $T$ on $(X_0,X_1)$ (i.e. $T:X_0+X_1\to X_0+X_1$ is bounded linear such that $T|_{X_i}:X_i\to X_i$ is also bounded for $i=0,1$), $T|_X:X\to X$ is also bounded.
    \end{itemize}
\end{defn}
\begin{defn}
    Let $\mathscr X$ be a Hausdorff topological space and let $\mathfrak{S}$ be a collection of quasi-Banach spaces $X\subseteq\mathscr X$, such that $X\hookrightarrow\mathscr X$ are all topological embeddings. 
    
    We say $Y$ is a \textit{(set-theoretical) interpolation space} from $\mathfrak S$, if there are $X_0,X_1\in\mathfrak S$ such that $Y$ is a set-theoretical interpolation of $(X_0,X_1)$. 
\end{defn}
In this way Corollary~\ref{Cor::Interpo} can be formulated to the following:
\begin{cor}\label{Cor::InterpoReform2}
    Let $0<p,q\le\infty$ and $s\in\R$ such that $p\neq q$. There are no  $0<p_0,p_1,q_0,q_1\le\infty$, $s_0,s_1\in\R$ such that $\Fs_{pq}^s(\R^n)$ is a set-theoretical interpolation space of $(\Bs_{p_0q_0}^{s_0}(\R^n),\Bs_{p_1q_1}^{s_1}(\R^n))$.
    
\end{cor}
\begin{proof}
    The operator $T$ in Theorem~\ref{Thm::TBdd} is an admissible operator of $(\Bs_{p_0q_0}^{s_0}(\R^n),\Bs_{p_1q_1}^{s_1}(\R^n))$. However $T(\Fs_{pq}^s(\R^n))\not\subset\Fs_{pq}^s(\R^n)$. Therefore by definition $\Fs_{pq}^s(\R^n)$ is not a set-theoretical interpolation space of $(\Bs_{p_0q_0}^{s_0}(\R^n),\Bs_{p_1q_1}^{s_1}(\R^n))$.
\end{proof}

\section{A Short Notes on Structured Banach Spaces and Open Question}\label{Section::StrBana}

In this section we briefly recall the approach in \cite{KunstmannInterpolation} and purpose an open question in this framework. A special thanks to Emiel Lorist for the discussion of this topic.

A \textit{structured Banach space} is a triple $\mathcal X=(X,J,E)$ where $X$ is a Banach space, $E$ is a Banach function space (see \cite[Definition~2.1]{KunstmannInterpolation} for details) and $J:X\to E$ is a linear isometry map. Let $1\le q\le\infty$, we say a bounded linear map $S:X\to X$ is \textit{$\ell^q$-bounded with respect to $(J,E)$}, if   $$\exists C>0\quad\text{such that}\quad\bigg\|\Big(\sum_{k=1}^\infty|JSf_k|^q\Big)^{1/q}\bigg\|_{E}\le C\bigg\|\Big(\sum_{k=1}^\infty|Jf_k|^q\Big)^{1/q}\bigg\|_{E},\qquad\forall f_1,f_2,\dots\in X.$$
For such $S$ we also say that $S:\mathcal X\to\mathcal X$ is $\ell^q$-bounded.

Let $\mathcal X_i=(X_i,J_i,E_i)$, $i=0,1$ be two structured Banach spaces with $(X_0,X_1)$ being a compatible pair. For $1\le q\le\infty$ and $0<\theta<1$, the \textit{$\ell^q$-interpolation space} $(\mathcal X_0,\mathcal X_1)_{\theta,\ell^q}$ is a subspace of $X_0+X_1$ given by (see \cite[Definition~2.1 and Theorem~3.1]{KunstmannInterpolation})
\begin{equation*}
    \|f\|_{(\mathcal X_0,\mathcal X_1)_{\theta,\ell^q}}:=\inf_{\substack{(f_j)_{j\in\Z}\subset X_0\cap X_1\text{ such that }\\f=\sum_{j\in\Z}f_j\text{ converges in }X_0+X_1}}\big(\big\|(2^{-\theta j}J_0f_j)_{j\in\Z}\big\|_{E_0(\ell^q)}+\big\|(2^{(1-\theta)j}J_1f_j)_{j\in\Z}\big\|_{E_1(\ell^q)}\big).
\end{equation*}

The corresponding interpolation theorem \cite[Theorem~2.11]{KunstmannInterpolation} says that for structured Banach spaces $\mathcal X_i=(X_i,J_i,E_i)$, $\mathcal Y_i=(Y_i,K_i,F_i)$ ($i=0,1$) such that $(X_0,X_1)$, $(Y_0,Y_1)$ are compatible, if $S:X_0+X_1\to Y_0+Y_1$ is a linear operator such that $S:\mathcal X_i\to \mathcal Y_i$ is $\ell^q$-bounded, then $S:(\mathcal X_0,\mathcal X_1)_{\theta,\ell^q}\to(\mathcal Y_0,\mathcal Y_1)_{\theta,\ell^q}$ is bounded for $0<\theta<1$.

\begin{rmk}\label{Rmk::StrBana::Exam}
    The following are typical examples of structured Banach spaces:
\begin{enumerate}[(i)]
    \item The Bessel potential $(I-\Delta)^{s/2}:H^{s,p}(\R^n)\to L^p(\R^n)$ is isomorphism for $s\in\R$ and $1<p<\infty$. Therefore $(H^{s,p}(\R^n),(I-\Delta)^{s/2},L^p(\R^n))$ defines a structured Banach space. Similarly for the homogeneous case $(\dot H^{s,p}(\R^n),(-\Delta)^{s/2},L^p(\R^n))$ is also a structured Banach space.
    \item\label{Item::StrBana::Exam::Besov} For $s\in\R$ and a Littlewood-Paley family $\phi=(\phi_j)_{j=0}^\infty$, define 
    $$J^s_\phi f(j,x):=2^{js}\phi_j\ast f(x),\qquad(j,x)\in\Z_{\ge0}\times\R^n.$$ Then for $1\le p,q\le\infty$, $(\Bs_{pq}^s(\R^n),J^s_\phi,\ell^q(\Z_{\ge0};L^p(\R^n)))$ is a structured Banach space.
    
Indeed $\ell^q(\Z_{\ge 0};L^p(\R^n))$ is a Banach function space over $\Z_{\ge0}\times\R^n$ equipped with the natural product of counting and Lebesgue measures.
\end{enumerate}
\end{rmk}

Let $1<p<\infty$, $1\le q\le\infty$ and $s\in\R$. By \cite[Proposition~5.1]{KunstmannInterpolation} with $A=(-\Delta)^{\frac s2}$ we get for $0<\theta<1$, $\big((L^p(\R^n),I,L^p),(\dot H^{s,p}(\R^n),(-\Delta)^{\frac s2},L^p)\big)_{\theta,\ell^q}=\dot\Fs_{pq}^{\theta s}(\R^n)$, which is a $\ell^q$-interpolation on homogeneous Triebel-Lizorkin space. By simple modification from its proof one can show that:
\begin{equation}\label{Eqn::StrBana::FfromB}
    \Big(\big(L^p(\R^n),I,L^p(\R^n)\big),\big( H^{s,p}(\R^n),(I-\Delta)^{\frac s2},L^p(\R^n)\big)\Big)_{\theta,\ell^q}=\Fs_{pq}^{\theta s}(\R^n),\quad 0<\theta<1.
\end{equation}
We leave the proof of \eqref{Eqn::StrBana::FfromB} to readers.

In our case $T$ in Theorem~\ref{Thm::TBdd} is not bounded in $\Fs_{pq}^s$ when $p\neq q$. As a result:
\begin{lem}
    The $T$ defined in Theorem~\ref{Thm::TBdd} is not $\ell^q$-bounded on $(H^{s,p}(\R^n),(I-\Delta)^{s/2},L^p(\R^n))$ for every $1\le q\le\infty$, $1<p<\infty$ and $s\in\R$, unless $p=q=2$.
\end{lem}

\begin{proof}[Sketch]
    Recall from Remark~\ref{Rmk::WspHsp}, $T$ is unbounded on $H^{s,p}$ unless $p=2$. It remains to prove the case $p=2\neq q$. 
    
    Define $\vec f_{a,u}:=(a_k\tau_{-u_k})_{k=1}^\infty$ for $a=(a_k)_{k=1}^\infty\subset\R$ and $(u_k)_{k=1}^\infty\subset\R^n$. Clearly $\|(I-\Delta)^{\frac s2}\vec f_{a,u}\|_{L^2(\ell^q)}\approx_s\|\vec f\|_{L^2(\ell^q)}$.
    When $q<2=p$, take $a\in\ell^2\backslash\ell^q$ and $u_k\equiv y_k$, we see that $\|\vec f\|_{L^2(\ell^q)}\lesssim\|a\|_{\ell^2}$ but $\|T\vec f\|_{L^2(\ell^q)}\gtrsim\|a\|_{\ell^q}=\infty$.
    When $q>2=p$, take $a\in\ell^q\backslash\ell^2$ and $u_k\equiv0$, we see that $\|\vec f\|_{L^2(\ell^q)}\lesssim\|a\|_{\ell^q}$ but $\|T\vec f\|_{L^2(\ell^q)}\gtrsim\|a\|_{\ell^2}=\infty$.
\end{proof}

It is natural to ask whether we can get a general interpolation result to \eqref{Eqn::StrBana::FfromB} that obtains Triebel-Lizorkin spaces:
\begin{ques}\label{Ques::StrBana}
    Let $0<p,q\le\infty$, $s_0\neq s_1$ and $0<\theta<1$. Can we find Banach function spaces $E_j=E_j(p,q,s_0,s_1,\theta)$ and isometric mapping $J_j=J_j(p,q,s_0,s_1,\theta):\Bs_{pp}^{s_j}(\R^n)\to E_j$ for $j=0,1$ such that the following holds?
    \begin{equation*}
        \Big(\big(\Bs_{pp}^{s_0}(\R^n),J_0,E_0\big),\big( \Bs_{pp}^{s_1}(\R^n),J_1,E_1\big)\Big)_{\theta,\ell^q}=\Fs_{pq}^{\theta s}(\R^n),\quad 0<\theta<1.
    \end{equation*}

    If such $J_0,J_1,E_0,E_1$ do not exist, can we still obtain $\Fs_{pq}^{\theta s}(\R^n)$ from the framework of sequential structures given in \cite{LindermulderLoristInterpolation}?

    If such $J_0,J_1,E_0,E_1$ always exist, can we make $(J_j,E_j)$ depend only on $p$ and $s_j$?
\end{ques}

Notice that in both \cite{KunstmannInterpolation} and \cite{LindermulderLoristInterpolation} only Banach spaces are discussed. One may need to give slight modification so that they work for quasi-Banach spaces when $p$ or $q<1$.

If one wants a positive answer to Question~\ref{Ques::StrBana}, the structure $J^s_\phi$ in Remark~\ref{Rmk::StrBana::Exam}~\ref{Item::StrBana::Exam::Besov} does not work.
\begin{lem}
    For $1\le p,q,r\le\infty$, $T$ is $\ell^r$-bounded with respect to the structure $(\Bs_{pq}^s,J^s_\phi,\ell^q(L^p))$. 
    
    In particular, for every $r\in[1,\infty]$, elements in $\{\Fs_{pq}^s(\R^n):s\in\R,\ p,q\in[1,\infty],\ p\neq q\}$ cannot be $\ell^r$-interpolated from the structured Besov spaces $\{(\Bs_{pq}^s,J^s_\phi,\ell^q(L^p)):s\in\R,\ p,q\in[1,\infty]\}$.
\end{lem}
This is likely to be true for $p,q$ or $r<1$, provided that we have a version interpolation for structured quasi-Banach spaces.
\begin{proof}
    Indeed $T$ is $\ell^r$-bounded in $(\Bs_{pq}^s,J^s_\phi,\ell^q(L^p))$ if and only if $T:\Bs_{pq}^s(\R^n;\ell^r)\to\Bs_{pq}^s(\R^n;\ell^r)$.

    On the other hand, we have $\|\phi_j\ast\phi_k\ast\vec f\|_{L^p(\ell^r)}\le\|\phi_k\ast \vec f\|_{L^p(\ell^r)}$ for every $\vec f\in\Ss'(\R^n;\ell^r)$, since $\|\phi_j\ast\vec g\|_{L^p(\ell^r)}\le\|\vec g\|_{L^p(\ell^r)}$ holds via Young's inequality. The boundedness $T$ on $\Bs_{pq}^s(\ell^r)$ then follows from repeating \eqref{Eqn::Bsbdd::MainIneqn} with $f$ replacing by $\vec f$.

    Let $s\in\R$ and $p,q\in[1,\infty]$ such that $p\neq q$. Suppose by contrast that there are $s_0,s_1\in\R$, $p_0,p_1,q_0,q_1,r\in[1,\infty]$ and $\theta\in(0,1)$ such that $\Fs_{pq}^s=((\Bs_{p_0q_0}^{s_0},J^{s_0}_\phi),(\Bs_{p_1q_1}^{s_1},J^{s_1}_\phi))_{\theta,\ell^r}$, then $T$ will be bounded on $\Fs_{pq}^s$. This  contradicts to the Theorem~\ref{Thm::TBdd}~\ref{Item::TBdd::Fs}.
\end{proof}

\begin{ack}
    The author thanks Jan Lang for suggesting an ongoing collaboration project that led to this paper. The author thanks Emiel Lorist and Winfried Sickel for some helpful discussion.

    The author also acknowledge the travel funding from the AIM Fourier restriction community.
\end{ack}
\bibliographystyle{amsalpha}
\small\bibliography{reference}
\end{document}